\documentclass[11pt]{amsart}

\usepackage{amssymb,amsmath}
\usepackage{bbm}
\usepackage{a4wide}

\newtheorem{theorem}{Theorem}
\newtheorem{prop}{Proposition}
\newtheorem{lemma}{Lemma}
\newtheorem{coro}{Corollary}
\newtheorem{fact}{Fact}
 
{
\theoremstyle{definition}
\newtheorem{defin}{Definition}
\newtheorem{remark}{Remark}
\newtheorem{example}{Example}
}

\newcommand{\ts}{\hspace{0.5pt}}
\newcommand{\CC}{\mathbb{C}\ts}
\newcommand{\RR}{\mathbb{R}\ts}
\newcommand{\QQ}{\mathbb{Q}\ts}
\newcommand{\ZZ}{\mathbb{Z}}
\newcommand{\TT}{\mathbb{T}}
\newcommand{\NN}{\mathbb{N}}

\newcommand{\one}{\mathbbm{1}}
\newcommand{\sph}{\mathbb{S}}
\newcommand{\up}{\widetilde{\!\!\!\hphantom{m}}}

\newcommand{\dd}{{\rm d}}
\newcommand{\exend}{\hfill$\Diamond$}

\DeclareMathOperator{\lcm}{lcm\ts}
\DeclareMathOperator{\mgcd}{mgcd}
\DeclareMathOperator{\ord}{ord}
\DeclareMathOperator{\sgn}{sgn}
\DeclareMathOperator{\trace}{tr}
\DeclareMathOperator{\card}{card\ts}
\DeclareMathOperator{\GL}{GL}
\DeclareMathOperator{\PGL}{PGL}
\DeclareMathOperator{\SL}{SL}

\DeclareMathOperator{\fix}{Fix}
\DeclareMathOperator{\per}{per}
\DeclareMathOperator{\Mat}{Mat}

\begin{document}

\title[Periodic orbits of toral endomorphisms]
{Periodic orbits of linear endomorphisms \\[2mm]
on the $2$-torus and its lattices}

\author{Michael Baake}
\address{Fakult\"at f\"ur Mathematik, Universit\"at Bielefeld, \newline
\hspace*{\parindent}Postfach 100131, 33501 Bielefeld, Germany}
\email{mbaake@math.uni-bielefeld.de}
\urladdr{http://www.math.uni-bielefeld.de/baake}

\author{John A.~G.~Roberts}
\address{School of Mathematics and Statistics, 
University of New South Wales, \newline
\hspace*{\parindent}Sydney, NSW 2052, Australia}
\email{jag.roberts@unsw.edu.au}
\urladdr{http://www.maths.unsw.edu.au/\~{}jagr}

\author{Alfred Weiss}
\address{Department of Mathematical and Statistical Sciences,
University of Alberta, \newline
\hspace*{\parindent}Edmonton, AB, Canada T6G 2G1}
\email{aweiss@math.ualberta.ca}
\urladdr{http://www.math.ualberta.ca/Weiss\_A.html}

\begin{abstract} 
  Counting periodic orbits of endomorphisms on the $2$-torus is
  considered, with special focus on the relation between global and
  local aspects and between the dynamical zeta function on the torus
  and its analogue on finite lattices.  The situation on the lattices,
  up to local conjugacy, is completely determined by the determinant,
  the trace and a third invariant of the matrix defining the
  toral endomorphism.
\end{abstract}

\maketitle

\section{Introduction}

The iteration of a continuous mapping $T$ of a compact space
$\varOmega$ into itself provides an example of a $\ZZ$-action on
$\varOmega$ and an important discrete dynamical system, usually
written as $(\varOmega,T)$. When $\varOmega$ is a metric space, the
system $(\varOmega,T)$ is called \emph{chaotic} when the periodic
orbits of $T$ are dense in $\varOmega$ and when also a dense orbit
exists, see \cite{BBCDS} for details. In general, significant
information about $T$ is contained in the periodic orbits of $T$ and
in their distribution over $\varOmega$. Knowledge of the periodic
orbits can be used to detect characteristic properties of $T$. For
example, if $T'$ represents another continuous mapping of $\varOmega$,
then a necessary condition for $T$ and $T'$ to be topologically
conjugate is that they share the same number of periodic points of
each period (presuming these numbers are finite).

It is the aim of this paper to contribute to the structure of periodic
orbits and related issues of conjugacy for the case of endomorphisms
of the $2$-torus, represented as usual by the action (mod $1$) of an
integer matrix $M\in\Mat (2,\ZZ)$ on $\TT^2 \simeq \RR^2/\ZZ^2$. A
well-studied subclass consists of the toral automorphisms, represented
by elements of the group $\GL (2,\ZZ)$, being the subgroup of matrices
with determinant $\pm 1$ within the ring $\Mat (2,\ZZ)$. Particularly
important are the hyperbolic ones (meaning that no eigenvalue is on
the unit circle), which are often called \emph{cat maps}.  Since these
are expansive, all periodic point counts are finite
\cite[Thm.~5.26]{W}.  Hyperbolic toral automorphisms are also
topologically mixing and intrinsically ergodic, see \cite{KH,W}. By the
Bowen-Sinai theorem, compare \cite[Thm.~2.2]{DEI}, this has the
consequence that the integral of a continuous function over $\TT^2$
equals its average value over the points fixed by $M^m$ in the limit
as $m \rightarrow \infty$.

The topological entropy of a hyperbolic toral automorphism $M\in \GL
(2,\ZZ)$ is given by $\log\,\lvert \lambda_{\mathrm{max}}\rvert$,
where $\lambda_{\mathrm{max}}$ is the eigenvalue of $M$ with modulus
$>1$.  This is also the metric (or Kolmogorov-Sinai) entropy of $M$,
and completely determines the dynamics up to metric isomorphism,
compare \cite{AW}.  This does not imply topological conjugacy though,
and one important difference emerges from the periodic orbits, which
live on a set of measure $0$. Indeed, on $\TT^2$, it is well-known
that the periodic orbits of hyperbolic linear endomorphisms lie on the
invariant lattices given by the sets of rational points with a given
denominator $n$, also known as $n$-division points (see
Section~\ref{torus-lattices} for more).  One of our main themes in
this paper is the interplay between the periodic orbit statistics on a
certain lattice (which we call \emph{local statistics}) versus
periodic orbit statistics on the union of all lattices (which we call
\emph{global statistics}). What determines when two cat maps have the
same global statistics? What determines when two cat maps have the
same local statistics on a certain lattice or on all lattices?

At the outset, it is worth saying that there have been many
investigations by others into classifying the periodic orbits of cat
maps, spread over a diverse range of the mathematics and physics
literature, compare \cite{HB,PV,K,DF,G,BF}\footnote{These
  investigations have used a variety of techniques. One is tempted to
  say, corrupting a proverb used by Mark Twain and others
  (\texttt{http://www.worldwidewords.org/qa/qa-mor1.htm}):
  \emph{There's more than one way to skin a cat (map).}} and further
references given there.  One motivation for this has come from the
interest in spatial discretisations of dynamical systems, itself
motivated by computer (screen) realisations of continuous phase
spaces. The time of recurrence of a hyperbolic $M \in \GL (2,\ZZ)$ on
the toral rational lattice with denominator $n$ is denoted by $\per\ts
(M,n)$, where this is the least common multiple of the periods present
on the $n$-division points.  The dependence of $\per\ts (M,n)$ on $n$
and related lower and upper bounds have been addressed in \cite{DF,BF,
  KR2,Ku,FND,S}.  In particular, the surprisingly low value of
$\per\ts (M,n)$ for some high values of $n$ (which correspond to a
very fine rational discretisation of the torus) has been investigated
in \cite{DF,BF} where it is shown that $\per\ts (M,n) \le 3n$ (see
\cite{S} for refinements of this bound).

A strong motivation for studying cats maps on lattices comes from
quantum mechanics, compare \cite{HB,K,KM,KR,KR2,Ku,FND,DW}. As
described in these references, quantum cat maps and their
perturbations are built from (classical) cat maps and their
perturbations \emph{restricted to a particular rational lattice}
(called the Wigner lattice in this instance). For this reason,
properties of a cat map that manifest themselves only on some rational
lattices, but not on others, can induce properties of the
corresponding quantum cat map on some Wigner lattices, but again not
on others.  Important cases of this occur for symmetries or (time)
reversing symmetries of a cat map, these being automorphisms of the
torus that commute with the cat map, respectively conjugate it into
its inverse.  By way of illustration, it was shown in \cite{BRcat}
that the first hyperbolic toral automorphism $A\in\SL (2,\ZZ)$ which
is not conjugate to its inverse in $\GL (2,\ZZ)$ (which actually also
excludes topological conjugacy to its inverse, compare \cite{AP})
occurs for trace $20$.  However, the global absence of time-reversal
symmetry did not affect the statistics of the eigenvalues of the
quantum cat map built from this example \cite{KM}.  As explained
there, this phenomenon is due to the fact that the quantum cat map
retains (time) reversing symmetry because $A$ is conjugate to its
inverse mod $n$ for any $n$ and that the quantum cat map is
constructed from the reduction of $A$ mod $n$.  Significantly, the
conjugating matrix on each lattice depends on $n$, consistent with
there being no global reversor. Recently, there has been quite some
interest in dealing with this challenge of so-called pseudo-symmetries
of quantum maps that are not quantisations of symmetries of the cat
map on the torus, but instead are manifestations of local symmetries
of the cat map restricted to some lattice \cite{KM,KR,DW}.

When trying to sort cat maps by their global or local periodic orbit
statistics, conjugacy is highly relevant. Conjugacy of $\GL (2,\ZZ)$
matrices is another topic that has arisen in a broad variety of
contexts and has been considered by many (see \cite{T,Rade,ATW} and
references therein). Conjugacy is determined by a triple of
invariants, namely the determinant, the trace and one other invariant
which can be related to ideal classes, representation by binary
quadratic forms or topological properties \cite{ATW}. Conjugacy in
$\GL (2,\ZZ)$ can also be completely decided by using the amalgamated
free product structure of $\PGL (2,\ZZ)$, which attaches a finite
sequence of integers to each element which corresponds to its normal
form as a word in the generators of the amalgamated free product
\cite{BRcat}. Clearly, conjugate cat maps share both the same global
period statistics and the same local statistics on each rational
lattice (where the dynamics is conjugate via the localisation mod $n$
of the global conjugating matrix). Also, cat maps that are just
conjugate on a given rational lattice will share the same local
statistics on that lattice. Being able to decide global and local
conjugacies is thus clearly important, as the statistics is the same
for all elements of a conjugacy class. But if two cat maps share the
same local statistics on a given rational lattice, are they linearly
conjugate on that lattice and what determines this?

The results of this paper will go some way towards answering the
questions raised above. After recalling some well-known facts in
Section~\ref{sec-two}, we look at periodic orbit counts for integer
matrices in terms of zeta functions in Section~\ref{sec-three}. The
dynamical zeta function for the global counts is described by
Proposition~\ref{zeta1}, generalising a result of \cite{DEI}. We then
discuss a zeta function for the local periodic counts derived from the
action of an integer matrix on a rational lattice. Theorem~\ref{limit}
relates the global and local zeta functions in a suitable limit. This
is followed by an interpretation in terms of finite Abelian groups.
Section~\ref{sec-four} addresses the issue of local conjugacies of
linear endomorphisms on rational lattices. The \emph{matrix gcd},
which we define in Section~\ref{sec:mgcd}, turns out to be a key
quantity. It is preserved by $\GL (2,\ZZ)$ conjugacy, so it provides a
quick tool to see that two $\GL (2,\ZZ)$ matrices with different
matrix gcd are not conjugate on the torus.  On the other hand,
Theorem~\ref{main-theorem} and Corollary~\ref{main-coro} show that two
integer matrices that share the same determinant, trace and matrix gcd
are linearly conjugate on all rational lattices of the torus. As an
illustration of this result, consider our discussion above of quantum
cat maps and time-reversal symmetry. The fact that any $M \in \SL
(2,\ZZ)$ shares determinant, trace and matrix gcd with $M^{-1}$ means
that the two matrices are conjugate on \emph{all} rational lattices,
though not necessarily by matrices that derive from one and the same
matrix on the torus. This is nevertheless sufficient to guarantee that
the associated quantum cat map has time reversal symmetry.

\section{Notation and general setting}\label{sec-two}

Here, we describe our setting and recall some well-known facts,
tailored to the situation at hand. While we go along, we also
introduce our notation and establish further connections with
related topics in the recent literature.

\subsection{Counting orbits}

Consider a compact space $\varOmega$ and some (continuous) mapping $T$
of $\varOmega$ into itself. Let $\fix_m (T) := \{ x\in X \mid T^m x =
x\}$ be the set of fixed points of $T^m$.  Of particular interest are
the \emph{fixed point counts}, defined as
\begin{equation} \label{def-a}
     a_m \, := \, \card \{x\in \varOmega \mid T^m x=x\}
     \, = \,  \card (\fix_m (T))\ts ,
\end{equation}
which need not be finite in general. In many interesting cases,
including all expansive homeomorphisms, this is the case
though, including the toral endomorphisms without eigenvalues on the
unit circle. 

The quantity $a_m$ has the disadvantage that one keeps recounting the
contributions $a^{}_{\ell}$ for all $\ell | m$. Clearly, the fixed
points of \emph{genuine} order $m$ permit a partition into disjoint
cycles, each of length $m$. If $c_m$ is the number of such cycles, one
thus has the relation
\begin{equation} \label{a-from-c}
   a_m \, = \, \sum_{d\ts | m} d\, c_d \ts .
\end{equation}
An application of a standard inclusion-exclusion argument, here 
by means of the M\"obius inversion formula from elementary
number theory, results in the converse identity,
\begin{equation} \label{c-from-a}
    c_m \, = \, \frac{1}{m} \sum_{d\ts |  m} 
   \mu \big( \tfrac{m}{d}\big) \ts  a_d\ts ,
\end{equation}
where $\mu(k)$ is the M\"obius function, compare \cite{Hasse,PW} and
references therein for details.

\begin{remark}\label{exact}
  Recall from \cite{PW} that a sequence $(a_m)_{m\in\NN}$ of
  non-negative integers is termed \emph{exactly realised} when it is
  the sequence of fixed point counts of a (continuous) dynamical
  system. This happens if and only if the derived sequence
  $(c_m)_{m\in\NN}$ is a sequence of non-negative integers
  \cite[Lemma~2.1]{PW}; see \cite{ward-online} for interesting
  examples other than toral endomorphisms and \cite{NN} for recent
  extensions of the concept.  \exend
\end{remark}

{}For later use, we briefly summarise some properties of the fixed point
and orbit counts. Let $(a_m)_{m\in\NN}$ and $(c_m)_{m\in\NN}$ be a
matching pair of such sequences, hence related as in
Eqs.~\eqref{a-from-c} and \eqref{c-from-a}. The sequence of fixed
point counts is called \emph{periodic}, when an $n\in\NN$ exists so
that $a_{m+n} = a_{m}$ holds for all $m\in\NN$. The least $n$ with this
property is called the \emph{period} of the sequence
$(a_m)_{m\in\NN}$.  The following consequences are standard.
\begin{fact} \label{fact-one} 
  Let the non-negative integer sequences $(a_m)_{m\in\NN}$ and
  $(c_m)_{m\in\NN}$ satisfy Eq.~$\eqref{a-from-c}$. If $a_m$ is
  periodic with period $n\in\NN$, one has $c_m=0$ for all $m>n$.
  Conversely, if only finitely many orbit counts $c_m$ differ from
  $0$, $a_m$ is periodic, with period $n=\lcm \{m\in\NN\mid c_m\neq
  0\}$.  \qed
\end{fact}

In extension of the usual practice for automorphisms, compare
\cite{KH}, we call a toral endomorphism $M\in \Mat(2,\ZZ)$
\emph{hyperbolic} when it has no eigenvalue on the unit circle
$\sph^1$.  Recall that the standard $2$-torus is $\TT^2 \simeq \RR^2/
\ZZ^2$, where $\ZZ^2$ is the square lattice in the plane. It is a
compact Abelian group, which can be written as $\TT^2 := [0,1)^2$,
with addition defined mod $1$.
\begin{fact} \label{fact:rational}
    If $M\in\Mat(2,\ZZ)$ is hyperbolic, $\fix_{m} (M) \subset (\TT^2
  \!\ts\cap \QQ^2)$ holds for all $m\in\NN$.
\end{fact}
\begin{proof}
  This can be shown by the concrete argument used in \cite[Sec.~1.8,
  p.~44]{KH}.  In essence, for any $m\in\NN$, the equation $M^{m} x =
  x$ mod $1$ means $(M^{m}-\one)x = v$ for some integer vector $v$.
  Since $1$ is not an eigenvalue of $M^{m}$, the integer matrix
  $M^{m}-\one$ is invertible, with rational inverse.  Solving for $x$
  then gives the claim.
\end{proof}

This property motivates to look at periodic points of toral endomorphisms
on certain subsets of $\TT^2\cap\QQ^2$ as well.      

\subsection{Invariant lattices on the $2$-torus}\label{torus-lattices}
Recall that a \emph{lattice} in a locally compact Abelian group is a
co-compact discrete subgroup, such as $\ZZ^2$ in $\RR^2$. For $\TT^2$,
a lattice simply means a discrete subgroup of it, which is then a
\emph{finite} Abelian group. The relevant lattices on $\TT^2$ are
\begin{equation} \label{define-L}
   L_n \, := \, \left\{ \big( \tfrac{k}{n},\tfrac{\ell}{n}\big)
   \mid 0 \le k,\ell < n \right\}\ts , \quad \text{for } n\in\NN\ts ,
\end{equation}
also known as $n$-division points, because they are invariant under
toral endomorphisms.

Consider an integer matrix $M$ that acts on $\TT^2$ (meaning that it
acts via matrix multiplication, evaluated mod $1$). Clearly, one then
has $M L_n \subset L_n$, and interesting information on the orbit
structure of the toral endomorphism or automorphism ($\det (M)=\pm 1$)
is contained in the distribution of its orbits on $L_n$.
Alternatively, one can characterise $L_n$ via
\begin{equation} \label{char-L}
   L_n \, = \, \{ x\in\TT^2\mid nx=0 \; \mbox{\rm (mod $1$)}\}\ts .
\end{equation}

\begin{remark} \label{remark-one} 
  When looking at the action of $M$ on $L_n$ numerically, it is
  usually easier to replace $L_n$ by $\tilde{L}_n := \{ (k,\ell)\mid 0
  \le k,\ell < n \}$, with the equivalent action of $M$ defined mod
  $n$.  This also applies to various theoretical arguments involving
  modular arithmetic.  Consequently, we use $L_n$ (with action of $M$
  mod $1$) and $\tilde{L}_n$ (with action mod $n$) in parallel.
  \exend
\end{remark}

{}From now on, we use the abbreviation $\ZZ_n = \ZZ/n\ZZ$ for the
finite integer ring mod $n$, and $\ZZ_n^{\times}=\{1\le k\le n \mid
\gcd (k,n) = 1\}$ for its unit group.  Let now $n\ge 2$ be arbitrary,
but fixed.  If we have a matrix $A\in\GL (2,\ZZ)$, its reduction mod
$n$ is still invertible in $\Mat (2,\ZZ_n)$, and thus an element of
$\Mat (2,\ZZ_n)^{\times}$, the unit group within $\Mat (2,\ZZ_n)$.
This group is often also called $\GL (2,\ZZ_n)$, though its elements
need not have determinant $\pm 1$.

\begin{remark} \label{remark-localauto}
  A matrix $M\in\Mat (2,\ZZ)$ that is not a toral automorphism (so
  $\det(M)\ne \pm 1$) may still be invertible on a given lattice $L_n$
  (or on $\tilde{L}_n$), meaning that $M$ mod $n$ is an element of
  $\Mat (2,\ZZ_n)^{\times}$. This happens precisely when $\det (M) \in
  \ZZ$ is relatively prime to $n$, which is equivalent to $\det (M) \in
  \ZZ_{n}^{\times}$ (where $\det (M)$ is taken mod $n$).  \exend
\end{remark}

When $M$ is a toral automorphism, or a toral endomorphism with $\det
(M) \in \ZZ_{n}^{\times}$, the invertible action of $M$ on the finite
set $L_n$ induces a \emph{permutation} of the $n^2$ elements of $L_n$,
with one element ($x=0$) always fixed. So, the induced permutation
$\pi^{(n)}_{M}$ can either be viewed as an element of the symmetric
group $S_{n^2}$ or of $S_{n^2 - 1}$. The permutation is of finite
order, which must divide $(n^2 -1) !$ by Lagrange's theorem. The
actual order of $M$ on $L_n$ is given by
\begin{equation} \label{per-one}
  \ord\ts (M,n) \, := \;
  \gcd\ts \{ m\in\NN_{0} \mid M^m \equiv\one \; \bmod{n} \}\ts .
\end{equation}
Clearly, $\ord\ts (M,1)=1$ in this setting. When $M$ is not invertible
on $L_n$, the definition results in $\ord\ts (M,n)=0$; otherwise,
$\ord\ts (M,n)$ is the smallest $m\in\NN$ with $M^m=\one$ mod $n$.

By an application of Fact~\ref{fact-one}, the orbit count sequence on
$L_n$ has vanishing entries beyond index $m=\ord\ts (M,n)$.  It is
natural to define $a^{(n)}_{\ell} := \card \{ x\in L_n\mid M^{\ell} x
= x \;\mbox{\rm (mod $1$)} \}$ and
\[
   c^{(n)}_{\ell}  := \, \card
   \{ \mbox{cycles of $M$ on $L_n$ of length $\ell$} \} 
\]
as the induced fixed point and orbit counts on $L_n$, again mutually
related as in Eqs.~\eqref{a-from-c} and \eqref{c-from-a}.  It is not
hard to see that
\begin{equation} \label{perord-one}
  \ord\ts (M,n) \, = \;\per\ts (M,n) \, := \;
  \lcm\ts \{ m\in\NN \mid c_{m}^{(n)} \ne 0\}\ts ,
\end{equation}
meaning that $\per\ts (M,n)$ is the $\lcm$ of the lengths of the
cycles on $L_n$.  Clearly, we can only have $c^{(n)}_{m}\neq 0$ when
$m$ is a divisor of $\per\ts (M,n)$.

When $M$ is a toral endomorphism with $\gcd (\det (M),n) \ne 1$, it is
not invertible on the lattice $L_n$.  Consequently, not all points of
$L_n$ show up in periodic orbits of $M$, and one has `pretails' to the
periodic orbits (there is always at least one periodic orbit on
$L_n$). Two matrices on $L_n$ may thus share the same cycle structure
there, but show different pretail patterns. As the latter give rise to
a directed pseudo-graph on $L_n$ (where `pseudo' simply refers to the
possibility of directed loops at a vertex), with the points of $L_n$
as vertices and the directed edges derived from the action of $M$, it
is reasonable to coin an adequate definition that covers both the case
when $M$ is invertible on $L_n$ and when it is not.

\begin{defin} \label{defn-local}
  We say that two matrices $M$ and $M^{\ts\prime}$ have the same
  \emph{local statistics} on $L_n$ when the induced directed
  pseudo-graphs on $L_n$ are isomorphic as graphs. When $M$ and
  $M^{\ts\prime}$ are invertible on $L_n$, this is equivalent to
  saying that they have the same periodic orbit counts.  Otherwise,
  this is equivalent to saying that they have the same periodic orbit
  counts plus isomorphic pretails between corresponding periodic
  orbits of the same length.
\end{defin}

This definition shows that two matrices with different pretail
structures for equal orbit counts would \emph{not} have the same local
statistics.  Note that, when $M$ and $M^{\ts\prime}$ are invertible on
$L_n$, they have the same local statistics if and only if their
associated permutations $\pi^{(n)}_{M}$ and
$\pi^{(n)}_{M^{\ts\prime}}$ are conjugate in $S_{n^2}$.

\subsection{Powers of $2\!\times\!2$-matrices}

Let $M\in\Mat (2,\CC)$ be a non-singular matrix, with $D:=\det (M)\neq
0$ and $T:=\trace (M)$. Define a two-sided sequence of (possibly
complex) numbers $p_m$ by the initial conditions $p^{}_{-1} = -1/D$
and $p^{}_{0}=0$ together with the recursion
\begin{equation} \label{recursion}
  \begin{split}
     p_{m+1} & \,  = \, T\ts p_{m} - D\ts p_{m-1}\ts ,
      \quad \text{for $m\ge 0$,}   \\ 
     p_{m-1} & \,  = \, \frac{1}{D}\ts (T\ts p_{m} - p_{m+1})\ts ,
      \quad \text{for $m\le -1$.}
  \end{split}
\end{equation}
This way, as $D\neq 0$, $p_m$ is uniquely defined for all $m\in\ZZ$.
Note that the sequence $(p_m)_{m\in\ZZ}$ depends only on the
determinant and the trace of $M$. When $M\in\Mat (2,\ZZ)$, one has
$p_m\in\QQ$, and $p_m\in\ZZ$ for $m\ge 0$. When $M\in\GL (2,\ZZ)$, all
$p_m$ are integers.

Let us first note an interesting property, which follows from a
straight-forward induction argument (in two directions).
\begin{fact} \label{quadratic-relation}
  The two-sided sequence of rational numbers defined
  by the recursion $\eqref{recursion}$ satisfies the
  relation $p_{m}^{2} - p_{m+1}^{}\ts p_{m-1}^{} = D^{m-1}_{}$,
  for all $m\in\ZZ$.  \qed
\end{fact}

\begin{lemma} \label{iterates}
  Let $M\in\Mat (2,\CC)$ be non-singular.
  For $m\in\ZZ$, one has the relation
 \[
    M^m \, = \, p_m\, M - D\ts p_{m-1}\, \one \ts .
\]
  In particular, one has $M^{-1} = \frac{1}{D}\ts (T\ts\one - M)$.
\end{lemma} 
\begin{proof}
The initial conditions $p^{}_{-1}=-1/D$ and $p^{}_{0}=0$ imply that
the relation $M^0=\one$ is matched. Let us first look at the positive
powers of $M$.  Assuming the claim to hold for some integer $m\ge 0$,
we can use the Cayley-Hamilton theorem to proceed inductively. Indeed,
from $M^2 = \trace(M)\, M - \det (M)\, \one = T\ts M - D\,\one$, one
finds with \eqref{recursion} that
\[
  \begin{split}
   M^{m+1} & \, = \, M^m \, M \, = \,
   p_m\ts M^2 - D\ts p_{m-1}\ts M \\
   & \, = \, (T\ts p_m - D\ts p_{m-1}) M - D\ts p_m\ts\one
   \, = \, p_{m+1}\ts M - D\ts p_m\ts\one\ts ,
  \end{split}
\]
which establishes the claim for all $m\ge 0$. The special relation for
$M^{-1}$ is just a reformulation of the inverse of a
$2\!\times\!2$-matrix, while the statement about the negative powers
follows from another induction argument, using the reverse recursion
for the $p_m$ with negative index.
\end{proof}

\begin{remark} \label{zero-det} 
  When $M\in\Mat (2,\CC)$ is \emph{singular}, so $D=\det (M)=0$, one
  can still meaningfully define the numbers $p_m$ for all $m\ge 1$. In
  fact, they are then simply given by $p_m = T^{m-1}$, with $T=\trace
  (M)$.  The formula from Lemma~\ref{iterates} simplifies to $M^m =
  p_m\ts M = T^{m-1}\ts M$, which is valid for all $m\ge 1$, while
  Fact~\ref{quadratic-relation} remains true for all $m\ge 2$.
  The numbers $p_m$ are particularly useful to determine the
  periods $\ord\ts (M,n)$ of Eq.~\eqref{per-one}.
\exend  
\end{remark}

\section{Dynamical zeta functions and periodic orbit statistics}
\label{sec-three}

To deal with combinatorial quantities such as the fixed point counts
$a_m$, it is advantageous to employ generating functions. They provide
a nice encapsulation of these numbers and permit the derivation of
several asymptotic properties as well. Here, the concept of a
\emph{dynamical zeta function}, compare \cite{R}, is usually most
appropriate.  Consequently, given a matrix $M\in\Mat (2,\ZZ)$, we set
\begin{equation} \label{define-zeta}
   \zeta^{}_{M} (t) \, := \; \exp\Bigl( \sum_{m = 1}^{\infty}
        \frac{a_m}{m}\, t^m \Bigr),
\end{equation}
where, from now on, $a_m := \card \{ x\in\fix_m (M)\mid x \text{ is
isolated} \}$ is the number of \emph{isolated} fixed points of $M^m$.
We say more about this below.

\begin{remark}
The ordinary power series generating function for the counts $a_m$ can
be calculated from $\zeta^{}_{M} (t)$ as $\sum_{m\ge 1} a_m\ts t^m =
t\ts\frac{\dd}{\dd t} \log \big(\zeta^{}_{M} (t)\big)$. The
significance of the formulation used in Eq.~\eqref{define-zeta}
follows from the fact that it has a unique Euler product decomposition
\cite{R} as
\begin{equation} \label{euler1}
   \frac{1}{\zeta^{}_{M} (t)} \;\, = \!
   \prod_{\text{cycles } \mathcal{C}}
   \big( 1- t^{\lvert \mathcal{C}\rvert}\big) \; = \,
   \prod_{m\ge 1}\, \big(1-t^m \big)^{c_m},
\end{equation}
where $\lvert \mathcal{C}\rvert$ stands for the length of the cycle
$\mathcal{C}$ and $c_m$ is now the number of \emph{isolated} cycles
of $M$ on $\TT^2$ of length $m$, as determined from Formula
\eqref{c-from-a}. Consequently, the role of cycles in dynamics is
similar to that of primes in elementary number theory.
\exend
\end{remark}

\subsection{Global considerations}

Let $M\in\Mat (2,\ZZ)$ and observe that $\fix_m (M)$ is an Abelian
group. In fact, it is a closed subgroup of $\TT^{2}$.  Consequently,
when $1$ happens to be an eigenvalue of $M^m$ for some $m$, there is a
continuous subgroup of fixed points of $M^m$, and one cannot have any
isolated fixed points in addition, due to the group structure of
$\fix_m (M)$, see \cite{BHP} for details. For a given toral
endmorphism, any set $\fix_m (M)$ is thus either finite (when $0$ is
an isolated fixed point) or a continuous submanifold of $\TT^2$.

With hindsight, this motivates our definition in \eqref{define-zeta},
and we now need to have an explicit formula for the number $a_m$ of
isolated fixed points of $M\in\Mat (2,\ZZ)$ on $\TT^2$. This is possible
via a standard argument that involves areas of fundamental domains
of planar lattices, compare \cite{BHP,DEI}. The subtorus case
is treated in \cite[Appendix]{BHP}.
\begin{fact}  \label{count-formula}
  If $M\in\Mat (2,\ZZ)$ is an arbitrary integer matrix,
  the number of isolated fixed points of $M^m$ on $\TT^{2}$ 
  is given by
\[
      a_m \, = \, \bigl\lvert\det(M^m - \one)\bigr\rvert \ts ,
\]
  which is valid for all $m\in\NN$.  In particular, all these counts
  $a_m$ and the corresponding cycle numbers $c_m$ are finite. 
  Moreover, whenever $a_m=0$, no isolated fixed points exist,
  and one has subtori of fixed points instead.   \qed
\end{fact}

One can express $a_m$ in terms of the numbers $p_m$ from
\eqref{recursion}.  Indeed, when $D=\det (M)\neq 0$, by inserting
the formula from Lemma~\ref{iterates} and by also using
Fact~\ref{quadratic-relation}, one finds
\begin{equation}
   \det (M^m - \one) \, = \, -p_{m+1} + D\ts p_{m-1}
     + D^m + 1 \ts ,
\end{equation}
which is valid for all $m\ge 1$. By Remark~\ref{zero-det}, one can
check that inserting $D=0$ gives the correct formula also for 
$\det(M)=0$, namely $\det (M^m-\one) = 1 - T^m$, with $T=\trace
(M)$. With these formulae, one can determine the dynamical zeta
functions of the isolated fixed points for many cases explicitly,
compare \cite[Sec.~I.4]{Smale} for an alternative approach.

\begin{example}\label{ex:divzeta}
  Let us start with $M\in\Mat (2,\ZZ)$ being singular, and of trace
  $T$. One finds $\zeta^{}_{M} (t) = 1/(1-t)$ for $T=0$ and
\begin{equation} \label{zeta-for-zero-det}
   \zeta^{}_{M} (t) \, = \, 
   \frac{1-\sgn (T)\ts t}{1-\lvert T \rvert\ts t}
\end{equation}
for $T\neq 0$ (note the special role of $T = 1$, with $\zeta^{}_{M}
(t) = 1$, for the existence of subtori of solutions). The explicit
derivation follows from Remark~\ref{zero-det} and the standard 
power series identity (with $\varrho =1$ as radius of convergence)
\begin{equation}\label{log-series}
    \log (1-z) \, = \,
    - \sum_{m=1}^{\infty} \frac{z^m}{m}\ts .
\end{equation}

When $M=k\one$ with $k\in\ZZ$, a simple calculation results in
\begin{equation} \label{k-times-id}
   \zeta^{}_{M} (t) \, = \, 
   \frac{(1-k\ts t)^2}{(1-t)\ts (1-k^2\ts t)} \ts ,
\end{equation}
which is also valid for $M=\left(\begin{smallmatrix} k & b \\ 0 & k
\end{smallmatrix}\right)$, with $b\in\ZZ$ arbitrary. If $k=1$, we are 
back to a case with subtori of fixed points, again reflected by
$\zeta^{}_{M} (t) =1$.

More generally, consider a non-singular upper triangular matrix of the
form $M=\left(\begin{smallmatrix} a & b \\ 0 & d
  \end{smallmatrix}\right)$ with $ad\neq 0$, so that $\delta = \sgn
(D) \neq 0$. Then, one finds the zeta function
\begin{equation} \label{zeta-for-upper}
   \zeta^{}_{M} (t) \, = \, 
   \frac{\det (\one - \delta t M)}{(1-\delta t)\ts (1-\delta D t)}
   \, = \, \frac{(1-\delta a t)\ts (1-\delta d t)}
   {(1-\delta t)\ts (1-\delta D t)} \ts ,
\end{equation}
again using standard manipulations of power series.
\exend
\end{example}

\subsection{Zeta functions of toral automorphisms}
For $M\in\GL (2,\ZZ)$, Fact~\ref{count-formula} permits a derivation
of the dynamical zeta function as follows, a special case of which was
also given in \cite{DEI}.

\begin{prop} \label{zeta1}
   Let\/ $M\in\GL (2,\ZZ)$ be hyperbolic, and define\/
   $\sigma=\sgn \bigl(\trace (M)\bigr)$. Then, 
   with the coefficients\/ $a_m = \card \{ x\in\TT^2\mid M^m x=x 
   \mbox{ \rm (mod $1$)}\}$, the dynamical zeta function\/ 
   \eqref{define-zeta} of\/  $M$ on\/ $\TT^2$ is given by
\[
    \zeta^{}_{M} (t)  \, = \,
    \frac{(1-\sigma t) (1-\sigma t\ts \det(M))}
    {\det (\one-\sigma t M )}  \, = \,
    \frac{(1-\sigma\ts t) (1-\sigma \det(M)\, t)}
    {1-\lvert\trace (M)\rvert\, t+\det(M)\, t^2} \ts .
\]
   In particular, $\zeta^{}_{M} (t)$ is a rational
   function, with numerator and denominator in\/ $\ZZ[t]$. The
   denominator is a quadratic polynomial that is irreducible over
   $\ZZ$.  Its zero $t_{\rm min}$ closest to $0$ gives the radius of
   convergence of $\zeta^{}_{M} (t)$, as a power series around\/ $0$,
   via $r_{\rm c} = \lvert t_{\rm  min} \rvert$.
\end{prop}
\begin{proof}
Recall first from \cite{R} that, for arbitrary matrices $A\in\Mat (n,\CC)$,
\begin{equation}\label{zeta-ruelle}
   \exp \Bigl(\sum_{k=1}^{\infty} 
        \frac{\trace (A^k)}{k}\, t^k \Bigr)
        \, = \, \frac{1}{\det (\one - t A)}\ts ,
\end{equation}
with convergence for $\lvert t \rvert < 1/\varrho$, where
$\varrho$ is the spectral radius of $A$.

If $M$ is hyperbolic, the general formula for the $a_m$ from 
Fact~\ref{count-formula} can be used to derive 
\[
    a_m \, = \, \sigma^m \bigl( \trace (M^m)
    - (1 + \det(M)^m )\bigr)
\]
by observing that the two eigenvalues of $A$ can be written as $\lambda$
and $\det(A)/\lambda$. For the detailed argument, one may assume
$\lvert\lambda\rvert>1$ and check the different cases. Note that a
hyperbolic toral automorphism is never of trace $0$.
 
The formula for the zeta function now follows from \eqref{define-zeta}
by inserting the expression for $a_m$ and using the relation
\eqref{zeta-ruelle} together with the power series identity
\eqref{log-series}. The statement on the nature of the rational
function is then clear. With $M\in\GL (2,\ZZ)$, the denominator only
factorises for $\trace (M)=0$, $\det(M)=-1$ or for $\trace (M)=\pm 2$,
$\det(M)=1$, both cases being impossible for hyperbolic matrices.  The
result on the radius of convergence is standard.
\end{proof}

\begin{example}\label{ex:golden}
Probably the best known hyperbolic toral automorphism is the one
induced by the `classic' or golden cat map
\[
   M \, = \, \begin{pmatrix} 0 & 1 \\ 1 & 1 \end{pmatrix}.
\]
It has $\det (M)=-1$ and is thus orientation reversing (sometimes, as in
\cite{KH}, its square is used instead). From
Proposition~\ref{zeta1} or from \cite{BHP}, one obtains
\[
   \zeta^{}_{M} (t)  \, = \, \frac{1-t^2}{1-t-t^2}
   \, = \prod_{m\ge 1} \big(1-t^m\big)^{-c_m}
\]
with $a_m = f_{m+1} + f_{m-1} - \big( 1 + (-1)^m\big)$ and $c_m$
according to Eq.~\eqref{c-from-a}, see also entries \texttt{A$\ts$001350}
and \texttt{A$\ts$060280} of \cite{online}. Here, $f_m$ are the Fibonacci
numbers, defined by the recursion $f_{m+1} = f_m + f_{m-1}$, for $m\ge
0$, together with the initial condition $f_0=0$ and $f_{-1}=1$. The
first few terms of the counts are given in Table~\ref{fib-tab}.  As an
aside, note that $\zeta^{}_{M} (t) = 1 + \sum_{m=0}^{\infty} f_m \ts
t^m$, and one has $M^m = f_m\, M + f_{m-1}\, \one$, the latter being valid
for all $m\in\ZZ$.
\exend
\end{example}

\begin{table}
\begin{tabular}{c|rrrrr|rrrrr|rrrrr}
$m$ &  1 &  2 &  3 &  4 &  5 &
       6 &  7 &  8 &  9 & 10 &
      11 & 12 & 13 & 14 & 15 \\
\hline
$a_m$ &   1 &   1 &   4 &   5 &   11 &
         16 &  29 &  45 &  76 &  121 &
        199 & 320 & 521 & 841 & 1364\rule[-1ex]{0ex}{3.5ex} \\
$c_m$ &   1 &   0 &   1 &   1 &    2 &
          2 &   4 &   5 &   8 &   11 &
         18 &  25 &  40 &  58 &   90
\end{tabular}
\bigskip
\caption{Fixed point and orbit counts for the golden cat map. 
\label{fib-tab}}
\end{table}

\begin{remark}\label{rem:pisano}
  In Example~\ref{ex:golden}, when taken mod $n$, the powers $M^m$ are
  periodic in $m$ with period $P_n$, which are known as the
  \emph{Pisano periods}, see entry \texttt{A$\ts$001175} of
  \cite{online}. The periods of the sequences
  $\bigl(a^{(n)}_{m}\bigr)_{m\in\NN}$ divide $P_n$, but can be
  smaller, as happens for $n=4$ (with $3$ versus $P_4=6$) or for $n=5$
  (with $4$ versus $P_5=20$).  \exend
\end{remark}

Let us return to our general discussion.
{}From Proposition~\ref{zeta1}, it is clear that two hyperbolic
$\GL (2,\ZZ)$-matrices with the same trace and determinant
possess the same dynamical zeta function, hence the same fixed point
counts. The converse is slightly more subtle.

\begin{coro} \label{coro-on-zeta}
 Let $M,M^{\prime}\in\GL (2,\ZZ)$ represent two hyperbolic toral 
 automorphisms which have the same fixed point counts. Then,
 $\zeta^{}_{M} (t) = \zeta^{}_{M^{\prime}} (t)$. This implies
 $\det(M^{\prime}) = \det(M)$ and either\/
 $\trace (M^{\prime}) = \trace (M)$ or\/
 $\trace (M^{\prime}) = - \trace (M)$, the latter together with
 $\det(M)=-1$.
\end{coro}
\begin{proof}
  The claimed equality of the zeta functions is clear. As the
  denominator is irreducible over $\ZZ$ by Proposition~\ref{zeta1}, 
  we get $\det(M^{\prime}) = \det(M)$ and $\lvert\trace
  (M^{\prime})\rvert = \lvert\trace (M)\rvert$.  Equating the
  numerators results in the two possibilities stated.
\end{proof}

\begin{remark}
The second possibility of Corollary~\ref{coro-on-zeta} is realised
by any orientation reversing $\GL (2,\ZZ)$-matrix $M$ together with 
$M^{\ts\prime}=-M$. One can also check this property by an explicit
calculation, using Fact~\ref{count-formula} in conjunction with 
Lemma~\ref{iterates} and Fact~\ref{quadratic-relation}.

A concrete example is provided by the golden cat map
(or Fibonacci matrix) $M=\left(\begin{smallmatrix} 0 & 1 \\ 1 & 1
\end{smallmatrix}\right)$ of Example~\ref{ex:golden}.
Since $-\one$ is a lattice symmetry for all $L_n$, 
it is clear that also the local statistics of $M$ and $-M$
is the same on all lattices.
\exend
\end{remark}

The formula of Proposition~\ref{zeta1} does not necessarily hold for
the elliptic and parabolic elements of $\GL (2,\ZZ)$.  However, those
cases can be derived from the formulas given in
Example~\ref{ex:divzeta}, and result, in our setting, in the
dynamical zeta function for the \emph{isolated} fixed points.

\subsection{Generating functions on lattices}

Let us now consider a toral endomorphism $M$ on the lattice $L_n$, for
some $n\in\NN$. It is clear that $L_n$ is mapped onto itself under the
action of $M$ (mod $1$).  Alternatively, by Remark~\ref{remark-one},
we may consider $\tilde{L}_n$, and thus the action of the reduction of
$M$ mod $n$.  Recalling that $\card (L_n)=\card (\tilde{L}_n)=n^2$,
the following observation is obvious.
\begin{fact} \label{general-poly} Let $M\in\Mat (2,\ZZ)$ and $n\in\NN$
  be fixed. Then, only finitely many of the orbit counts
  $c^{(n)}_{\ell}$ on $L_n$ can be non-zero, and the dynamical 
  zeta function for the action of $M$ on $L_n$ has the property that
\[
      Z_n (t) \, := \, \frac{1}{\zeta^{(n)}_{M} (t)} \, = \,
      \prod_{\ell \ge 1} (1-t^\ell)^{c^{(n)}_{\ell}}   
\]
  is a finite product and defines a polynomial in $\ZZ[t]$ of 
  degree at most $n^2$.    \qed
\end{fact}  

Note that the degree of $Z_n (t)$ can indeed be less than $n^2$.  This
happens whenever the action of $M$ is non-invertible on $L_n$, which
manifests itself in the existence of pretails to the actual orbits.
Otherwise, Fact~\ref{general-poly} can be sharpened as follows.
\begin{prop} \label{zeta-is-poly} 
  Let $M\in\GL (2,\ZZ)$ or, more generally, let $\det (M)$ for
  $M\in\Mat (2,\ZZ)$ be coprime with $n$, meaning that the reduction
  of $M$ is an element of\/ $\Mat (2,\ZZ_n)^{\times} =
  \GL(2,\ZZ_{n})$. Then, the
  dynamical zeta function $\zeta^{(n)}_{M} (t)$ satisfies
\[
   \frac{1}{\zeta^{(n)}_{M} (t)} \, = \,
   \prod_{\ell\, | \per\ts (M,n)}\, 
   \bigl(1-t^{\ell}\bigr)^{c^{(n)}_{\ell}}
   \, =: \, Z_n (t)\ts ,
\]
  where $Z_n (t)\in\ZZ[t]$ has degree
  $n^2$. In particular, with $N=\ts\per\ts (M,n)$, one has
\[
   \sum_{\ell\, | N} \ell \,c^{(n)}_{\ell}
   \, = \, a^{(n)}_{N} \, = \, n^2 ,
\]
  and the minimal period of the sequence
  $\bigl(a^{(n)}_{m}\bigr)_{m\in\NN}$ is a divisor of $N$.  \qed
\end{prop}

Concerning the last statement, it is worthwhile to mention that the
minimal period of the fixed point counts on $L_n$ can actually be
smaller than $\per\ts (M,n)$, as we saw in Remark~\ref{rem:pisano}.
Note that $(1-t)\ts | \ts Z_n (t)$ for all $n\ge 1$, since $x=0$ is a
fixed point of $M$ on all $L_n$. Moreover, $m\ts | \ts n$ implies $Z_m
\ts | \ts Z_n$. This induces a partial order on the polynomials $Z_n$,
which permits us to consider the direct (or inductive) limit of them,
considered within the ring of formal power series, $\ZZ[[t]]$, compare
\cite[Sec.~1]{GJ}. In fact, one simply has
\[
   \varinjlim Z_n (t)
   \, = \, \lim_{n\to\infty} \lcm \{ Z_1 (t), Z_2(t), \ldots ,
   Z_n (t)\}\ts ,
\]
where $\lcm$ stands for the least common multiple, which is
well-defined here because the polynomial ring $\ZZ[t]$ is factorial,
see \cite{L}.  It is clear by construction that this limit must divide
$1/\zeta^{}_{M} (t)$, the latter written out as an infinite Euler
product and thus as an element of $\ZZ[[t]]$. In our setting, the
formal power series have positive radius of convergence, so that one
can also extract asymptotic properties of their coefficients by
standard tools from complex analysis, compare \cite{DEI,GJ} for
examples.

\begin{theorem} \label{limit}
  If $M\in\Mat (2,\ZZ)$ is hyperbolic, one has
  $\, \varinjlim Z_n (t)
      \, = \,1/\zeta^{}_{M} (t) $ .   
\end{theorem}
\begin{proof}
  When $M$ has no unimodular eigenvalue, Fact~\ref{count-formula}
  gives a formula for the number of \emph{all} fixed points of $M^m$,
  which is finite.  It is clear that the finitely many fixed points of
  $M^m$ are isolated. Viewing each polynomial $Z_{n} (t)$ as an
  element of $\ZZ[[t]]$, it is clear that  
\[
   \varinjlim Z_n (t) 
   \mid 1/\zeta^{}_{M} (t) \ts ,
\]
where $1/\zeta^{}_{M} (t)$ is written as the Euler product of
\eqref{euler1}, hence as an element of $\ZZ[[t]]$, which contains each
$Z_{n} (t)$ as a factor. It remains to show that each factor
$(1-t^m)^{c_m}$ of $1/\zeta^{}_{M} (t)$ divides $Z_n (t)$ for some
$n\in\NN$ (and then also for all multiples $n'$ of $n$).

Note that all fixed points of $M^m$ lie in $\QQ^2 \!\ts\cap\TT^2
\simeq \QQ^2 / \ZZ^2$ by Fact~\ref{fact:rational}.  Consequently,
there must be some $n=n(m)$ such that all these fixed points are in
$L_n$, which implies $(1-t^m)^{c_m} \ts | \ts Z_n (t)$. The claim now
follows from the general structure of the direct limit.
\end{proof}

\begin{remark}
The direct limit does not exist for all matrices in $\Mat
(2,\ZZ)$. This relates to the observation that endomorphisms with
subtori of fixed points of some order have no isolated fixed points of
the same order, while the intersections with the lattices $L_n$ are
still finite sets. In this situation, $\zeta^{}_{M} (t)$ encapsulates
the isolated fixed point counts only (if any), while $Z_n (t)$
gradually explores also the non-isolated ones.
\exend
\end{remark}

\subsection{Group theoretic interpretation}

Both the torus $\TT^2$ and its lattices $L_n$ are compact Abelian
groups (the latter even being finite). It is thus natural to also
expect some group theoretic interpretation of counting orbits of a
hyperbolic toral endomorphism $M$ on these groups.  Recall that
\begin{equation} \label{group-1} 
     A_m \, := \, \fix_m (M) \, = \,
    \{ x\in\TT^2 \mid M^m x = x \; \mbox{\rm (mod $1$)} \}
\end{equation}
is a subgroup of $\TT^2$, hence a \emph{finite} Abelian group of order
$a_m$ due to the assumption of hyperbolicity.  Knowing $a_m$, however,
generally tells us rather little about $A_m$ as a group. To improve on
this, we employ the elementary divisor theorem for finite Abelian
groups, compare \cite[Thm.~I.4.2]{CR} or \cite[Ch.~I.4.8]{L}. To this
end, consider
\[
    A^{(n)}_{m} \, := \, A^{}_{m} \cap L^{}_{n}
   \, = \,
   \{ x\in A^{}_{m} \mid nx=0 \; \text{(mod $1$)} \} \ts ,
\]
which defines a family of Abelian groups, with
$a^{(n)}_{m} = \card \big(A^{(n)}_{m} \big) \le a^{}_{m}$
and $A^{(n)}_{m} \subset A^{}_{m}$. In fact, since we assume here
that $M$ has no eigenvalues on $\sph^1$, one has
$A^{}_{m} = \bigcup_{n\in\NN} A^{(n)}_{m}$.

\begin{prop}
  The structure of the finite Abelian group $A_m$ of\/
  $\eqref{group-1}$ is completely determined by the numbers
  $a^{(n)}_{m}$ with $n\in\NN$.
\end{prop}

\begin{proof} Fix $m\in\NN$, and write $G=A_m$. Choose an
isomorphism
\[
    G\simeq\bigoplus_{i=1}^{s} \ZZ / p_{i}^{\ell_{i}}\ZZ \ts ,
\]    
which exists by the elementary divisor theorem, with $s\in\NN$; note
that the primes $p_i$ need not be distinct.  Set $G(n)=A_{m}^{(n)}$
and $g(n)=a_{m}^{(n)}$, where $g(1)=1$.  In view of the
elementary divisor theorem, it now suffices to show that, for
each prime power $p^r$ with $r\ge 1$, the power of $p$ in
\[
       \frac{g(p^r) \, g(p^r)}{g(p^{r-1}) \, g(p^{r+1})} \, = \,
       \frac{[G(p^r) : G(p^{r-1})]}{[G(p^{r+1}) : G(p^r)]}
\]
equals the number of indices $i$ so that $p_{i} = p$ and
$\ell_{i} = r$. This follows from $g(p^r) = p^t$ where
\[
    t = \sum_{i \ts :\ts p_i = p} \min\ts (r,\ell_i) \ts .
\]
This uniquely specifies all elementary divisors.
\end{proof}

\section{Global versus local conjugacy and orbit statistics}
\label{sec-four}

The determinant and the trace of an integer matrix are not changed
under $\GL (2,\QQ)$ conjugacy.  But for matrices $M\in\Mat (2,\ZZ)$,
it has been known for a long time that the determinant and the trace
are neither a sufficient nor a maximal set of invariants (this goes
back to contributions by Latimer, MacDuffee, Taussky and Rademacher --
see \cite{T,ATW,Rade} and references therein). There are various ways
of deciding $\GL(2,\ZZ)$-conjugacy, amounting to exploiting a third
and final conjugacy invariant \cite{AO,ATW,BRcat,Rade}. It is clear
that there are many interesting connections to class groups and class
numbers of quadratic number fields, see \cite{T} for details.

\begin{example}\label{old-ex}
Consider the two $\GL (2,\ZZ)$-matrices
\[
     M = \begin{pmatrix} 3 & 10 \\ 1 & 3 \end{pmatrix}
     \quad \text{and} \quad
     M^{\ts\prime} = \begin{pmatrix} 3 & 5 \\ 2 & 3 \end{pmatrix},
\]
which share $D=-1$ and $T=9$. One can check explicitly that
the integral matrices $X$ which satisfy $MX=XM^{\ts\prime}$ are 
integral linear combinations of $A$ and $B$, where
\[
     A = \begin{pmatrix} 0 & 5 \\ 1 & 0 \end{pmatrix}
     \quad \text{and} \quad
     B = \begin{pmatrix} 2 & 0 \\ 0 & 1 \end{pmatrix}.
\]
For each integer $n>2$, one can find an $X$ with $\det (X)$ coprime
to $n$, so that the reductions of $M$ and $M^{\ts\prime}$ mod $n$
are $\Mat (2,\ZZ_n)^{\times}$-conjugate. However, when taken mod $5$, 
$\det (X)$ is always congruent to $0$, $2$ or $3$. Consequently, no
$X$ exists with $\det (X)=\pm 1$, whence $M$ and $M^{\ts\prime}$
are not conjugate in $\GL (2,\ZZ)$.
 \exend
\end{example}

Obviously, two $\GL (2,\ZZ)$-conjugate hyperbolic toral automorphisms
do possess the same dynamical zeta function on $\TT^2$, equivalently
have the same sequence of fixed point counts.  They also have the same
local statistics on all lattices $L_n$. The latter statement follows
from the observation that the conjugating $\GL (2,\ZZ)$-matrix element
leaves all $L_n$ invariant and induces a local conjugacy on all of
them. This is a particular case of a result for endomorphisms
(recall Remark~\ref{remark-localauto} and its preceding paragraph):
\begin{fact} \label{global-to-local}
   Let $n\ge 2$ be an integer.
   If two matrices $M,M^{\ts\prime} \in\Mat (2,\ZZ)$ are $\GL
   (2,\ZZ)$-conjugate, their reductions mod $n$ are
   $\Mat (2,\ZZ_n)^{\times}$-conjugate.
\end{fact}
\begin{proof}
  By assumption, we have $M^{\ts\prime}=A\ts M A^{-1}$ for some $A\in
  \GL(2,\ZZ)$, which mod $n$ is turned into an equation of the same
  type within $\Mat (2,\ZZ_n)$, with $A\in \Mat (2,\ZZ_n)^{\times}$.
\end{proof}

Conversely, if two hyperbolic matrices $M, M^{\ts\prime}\in\Mat
(2,\ZZ)$ share the same fixed point counts on all lattices $L_n$, they
must also have the same fixed point counts on $\TT^2$. If $M,
M^{\ts\prime}\in\GL (2,\ZZ)$, Corollary~\ref{coro-on-zeta} implies
that they have the same determinant. They also have the same trace if
they are orientation-preserving (their traces may differ in sign if
they are orientation-reversing). Even in the orientation-preserving
case, the equivalence of local statistics for all $n$ does \emph{not}
imply $\GL(2,\ZZ)$-conjugacy. For instance, $M$ and
$M^{\ts\prime}=M^{-1}$ must have the same set of fixed point counts on
all lattices $L_n$, as $M^{-1}$ simply runs backwards through the
orbits of $M$. But a hyperbolic toral automorphism need not be
conjugate to its inverse:

\begin{example} \label{not-conjugate}
The two matrices
\begin{equation} \label{ex-non-rev}
   M \, = \, \begin{pmatrix} 4 & 9 \\ 7 & 16 \end{pmatrix}
   \quad \mbox{and} \quad
   M^{-1} \, = \, \begin{pmatrix} 16 & -9 \\ -7 & 4 \end{pmatrix} 
\end{equation}
with $D=1$, $T=20$ are \emph{not} conjugate within $\GL (2,\ZZ)$, as
one can check by an explicit calculation \cite{BRcat} (in fact, this
means they are not even topologically conjugate, see
\cite[Fact~1]{BRtorus}). Note that the companion matrix $C$ for both
$M$ and $M^{-1}$ is conjugate to its inverse:
\begin{equation} \label{comp}
   C \, = \, \begin{pmatrix} 0 & -1 \\ 1 & 20 \end{pmatrix}
   \quad \mbox{with inverse} \quad
   C^{-1} \, = \, \begin{pmatrix} 20 & 1 \\ -1 & 0 \end{pmatrix}
   \, = \, \begin{pmatrix} 0 & 1 \\ 1 & 0 \end{pmatrix}
   \begin{pmatrix} 0 & -1 \\ 1 & 20 \end{pmatrix} 
   \begin{pmatrix} 0 & 1 \\ 1 & 0 \end{pmatrix} .
\end{equation}
This means that $C$ and its inverse are in yet another $\GL (2,\ZZ)$
conjugacy class than $M$ and $M^{-1}$ (one can calculate that there
are altogether $5$ conjugacy classes for $D=1$ and $T=20$, compare
\cite[Ex.~17]{AO2}).  We thus see that the set of integer matrices
with the same local statistics on all lattices $L_n$ can encompass
more than one matrix conjugacy class on $\TT^2$.  \exend
\end{example}

Two conjugate matrices possess equivalent orbit structures, including
pretails of periodic orbits. In view of Remark~\ref{remark-one}, the
following local property is then clear.
\begin{fact}\label{conj-implies-stats}
  Let $n\in\NN$.
  When two integer matrices $M$ and $M^{\ts\prime}$ are
  $\Mat (2,\ZZ_{n})^{\times}$-conjugate, the corresponding toral
  endomorphisms have the same local
  statistics on the lattice $L_n$, in the sense of
  Definition~$\ref{defn-local}$.  \qed
\end{fact}

In the remainder of this section, we answer the question when two
integer matrices
\begin{equation} \label{two-matrices}
  M \, = \, \begin{pmatrix} a & b \\ c & d \\
    \end{pmatrix} \quad \text{and} \quad
  M{\ts}' \, = \, \begin{pmatrix} a{\ts}' & b{\ts}' \\ 
    c{\ts}' & d{\ts}' \\ \end{pmatrix}
\end{equation}
from $\Mat (2,\ZZ)$ are locally conjugate (mod $n$) for all $n$ and
hence possess the same local statistics (in the sense of
Definition~\ref{defn-local}) on all lattices $L_n$. This will only
depend on the determinant, the trace and a new invariant that we
introduce next.

\subsection{The matrix gcd}\label{sec:mgcd}
Consider a $2\!\times\! 2$-matrix $M$ as in \eqref{two-matrices}. 
\begin{defin}\label{def:mgcd}
   If $M\in\Mat (2,\ZZ)$, the quantity
\[
   \mgcd (M) \, := \, \gcd (b,c,d-a) \ts ,
\]
is called the \emph{matrix gcd} of $M$, or $\mgcd$ for
short\footnote{A generalisation of the mgcd to $n\!\times\! n$ integer
  matrices, $n \ge 2$, is the quantity $m$ of \cite{B}, which is used to
  describe the least normal subgroup in $\GL (n,\ZZ)$ containing a
  given element.}.  Here, we take the gcd to be a non-negative
integer, and set $\mgcd(M)=0$ when $b=c=d-a=0$.
\end{defin}
The last convention matches that of the ordinary gcd,
and is compatible with modular arithmetic. The
following consequence of this definition is obvious.

\begin{fact} \label{exclude}
  {}For $M\in\Mat (2,\ZZ)$, the following statements are equivalent:
\begin{itemize}
\item[{\rm (a)}] The matrix gcd satisfies $\mgcd(M)=0$.
\item[{\rm (b)}] $M=k\one$ for some $k\in\ZZ$.
\item[{\rm (c)}] The minimal polynomial of $M$ is of degree $1$.
\end{itemize}
Consequently, whenever $\mgcd(M)=r\in\NN$, $M$ cannot be a multiple
of the identity, and its characteristic and minimal polynomials
coincide.                \qed
\end{fact}

There are various other useful properties of the matrix gcd. 
It is immediate that
\begin{equation} \label{gcd-on-transneg}
\mgcd (-M) \, = \, \mgcd (M) \, = \, \mgcd (M^t)
\end{equation}
holds for arbitrary $M\in\Mat (2,\ZZ)$. Moreover,
with $k\in\ZZ$, one has 
\begin{equation} \label{more}
  \mgcd(k M) \, = \, k\ts\mgcd (M) \quad \text{and} \quad 
 \mgcd (M+k\ts\one) \, = \, \mgcd (M)\ts .
\end{equation}
Finally, if $M$ is invertible, its inverse is
$M^{-1} = \frac{1}{\det (M)} \left(\begin{smallmatrix} d & -b \\ -c &
a \end{smallmatrix} \right)$. Consequently, for all
matrices $M\in\GL (2,\ZZ)$, one has the relation
\begin{equation} \label{gcd-on-inverse}
   \mgcd (M^{-1}) \, = \, \mgcd (M)
\end{equation}
in addition to \eqref{gcd-on-transneg}.

Most significantly, the matrix gcd satisfies the following invariance
property, which can be seen as a consequence of the close relationship
to the theory of binary quadratic forms \cite{Z}.
\begin{lemma} \label{mat-gcd}
    If $M,M^{\ts\prime}\in\Mat (2,\ZZ)$ are two integer matrices that
    are conjugate via a $\GL (2,\ZZ)$-matrix, one has
    $\mgcd (M^{\ts\prime}) = \mgcd (M)$. In particular, the matrix gcd 
    of Definition~$\ref{def:mgcd}$ is constant on the conjugacy classes
    of\/ $\GL (2,\ZZ)$.
\end{lemma}
\begin{proof}
  Assume $M^{\ts\prime} = A\ts M A^{-1}$ with
  $A=\left(\begin{smallmatrix} \alpha & \beta \\ \gamma & \delta
    \end{smallmatrix}\right) \in\GL (2,\ZZ)$. One can check by a
  straight-forward calculation that this implies the linear equation
\[
    \begin{pmatrix} b{\ts}' \\ c{\ts}' \\ d{\ts}'-a{\ts}' 
    \end{pmatrix} \, = \;
    \frac{1}{\det (A)} \begin{pmatrix}
    \alpha^2 & -\beta^2 & \alpha\beta \\
    -\gamma^2 & \delta^2 & -\gamma\delta \\
    2\alpha\gamma & -2\beta\delta & \beta\gamma + \alpha\delta
    \end{pmatrix} \begin{pmatrix} b \\ c \\ d-a\end{pmatrix}
    \, =: \, N \begin{pmatrix} b \\ c \\ d-a\end{pmatrix} ,
\]
where $N=N(A)$ is an integer matrix because $\det(A)=\pm 1$. As the
primed quantities are then integer linear combinations of the unprimed
ones, $\gcd (b,c,d-a)$ divides $\gcd
(b{\ts}',c{\ts}',d{\ts}'-a{\ts}')$. It is easy to verify that $\det
(N)=1$, so that $N\in\GL (3,\ZZ)$ and $N^{-1}$ is an integer matrix as
well (this can also be seen directly from observing that $M=A^{-1}
M{\ts}' A$ and hence $N(A^{-1}) = (N(A))^{-1}$).  Consequently, our
previous argument implies that $\gcd
(b{\ts}',c{\ts}',d{\ts}'-a{\ts}')$ divides $\gcd (b,c,d-a)$ as well.
Within $\NN$, we thus obtain $\gcd (b{\ts}',c{\ts}',d{\ts}'-a{\ts}') =
\gcd (b,c,d-a)$ as stated.  This conclusion also holds when one gcd
(and then also the other) vanishes, adopting the usual convention that
$0\ts | \ts 0$.  The second claim is now obvious.
\end{proof}

\begin{remark} \label{mgcd-global}
  Two integer matrices with different $\mgcd$ cannot be $\GL
  (2,\ZZ)$-conjugate. Note, however, that $M$, $M^{-1}$ and $C$ of
  Example~\ref{not-conjugate} all have $\mgcd=1$, but are in distinct
  $\GL (2,\ZZ)$ conjugacy classes as discussed.  \exend
\end{remark}

\subsection{Local conjugacies and a binary quadratic form}

With a view to determining conjugacies on $L_n$, meaning conjugacy
via $\Mat (2,\ZZ_n)^{\times}$, consider the integer matrices
\begin{equation} \label{mat-defs}
    M \, = \, \begin{pmatrix} a & b \\ c & d \end{pmatrix}
    \quad \text{and} \quad
    C \, = \, \begin{pmatrix} 0 & -D \\ 1 & T \end{pmatrix}
\end{equation}
with $D=\det (M)$ and $T=\trace (M)$. Here, $C$ is the standard
companion matrix for the characteristic polynomial 
\begin{equation} \label{charpoly}
      x^2 - T\ts x + D
\end{equation}
of the matrix $M$. Let us assume that $M$ is \emph{not} a multiple of
$\one$.  To investigate possible conjugacies, let
$A=\left(\begin{smallmatrix} \alpha & \beta \\ \gamma &
    \delta\end{smallmatrix}\right)$ be another integer matrix.  It is
easy to check that
\begin{equation} \label{fund-cond}
  M A \, = \, A\ts C \quad \Longleftrightarrow \quad
  \begin{pmatrix} \beta \\ \delta \end{pmatrix}
  \, = \, M \begin{pmatrix} \alpha \\ \gamma \end{pmatrix} .
\end{equation}
Whenever this semi-conjugacy holds, one has
\begin{equation} \label{det-form}
  \det (A) \, = \, c\,\alpha^2 + (d-a)\ts \alpha\gamma
                   - b\,\gamma^2
           \, =: \, Q^{}_{M} (\alpha,\gamma)\ts ,
\end{equation}
which brings us in contact with the classic theory of binary quadratic
forms. 

In fact, the quadratic form $Q^{}_{M}$ is the key for a hierarchy of
conjugacies.  Clearly, $M$ and $C$ are $\GL (2,\QQ)$-conjugate if
$Q^{}_{M}$ represents \emph{some} $\mu\neq 0$.  This is true unless
$\mgcd (M)=0$, which relates back to Fact~\ref{exclude}.  Whenever
$Q^{}_{M}$ represents $1$ or $-1$ (which then automatically implies
$\alpha$ and $\beta$ to be coprime), the matrices $M$ and $C$ are $\GL
(2,\ZZ)$-conjugate (this is one of the approaches to $\GL
(2,\ZZ)$-conjugacy taken in \cite{ATW,BRcat}).  When $Q^{}_{M}$
represents some $\mu\in\NN$, and when $n\in\NN$ is another integer
that is relatively prime with $\mu$, one has $\mu\in\ZZ_{n}^{\times}$,
wherefore $A$ is invertible in $\Mat(2,\ZZ_n)$, compare
Remark~\ref{remark-localauto}, and the reductions of $M$ and $C$ mod
$n$ are $\Mat (2,\ZZ_n)^{\times}$-conjugate.

The discriminant of the binary quadratic form $Q^{}_{M}$ of
\eqref{det-form} is
\begin{equation} \label{qdiscrim}
\varDelta = (d-a)^2 + 4 bc = (a+d)^2 - 4 (ad-bc) = T^2 - 4D.
\end{equation} 
The form $Q^{}_{M}$ is called \emph{primitive} when $\gcd(b,c,d-a)=1$,
which means $\mgcd (M)=1$. Moreover, a representation $k=Q^{}_{M}
(\alpha,\gamma)$ is called primitive (or proper) when $\gcd
(\alpha,\gamma)=1$.  We need the following fundamental result from the
theory of binary quadratic forms, compare \cite[Prop.~4.2]{Buell}.
\begin{fact} \label{prim-rep} 
  Let $n\in\NN$ be arbitrary, but fixed.  If the binary quadratic form
  $Q^{}_{M}$ is primitive, it can primitively represent some integer $k$
  that is relatively prime to $n$.  \qed
\end{fact}

Indeed, when $\alpha$ is the product of all primes that divide $n$,
but not $b$, and $\gamma$ the product of all primes that divide $n$
and $b$, but not $c$, one sees that $k=Q^{}_{M} (\alpha,\gamma)$ is an
integer that is relatively prime to $n$.

\begin{remark} \label{eigenvalue-discrim} When the quadratic form
  $Q^{}_{M}$ fails to be primitive, its discriminant $\!\varDelta$ is
  the product of $(\mgcd(M))^2$ with the discriminant $\!\varDelta'$
  of the `primitive part' of $Q^{}_{M}$. Using \eqref{qdiscrim}, this
  relates to the following property of the eigenvalues
  $\lambda^{}_{\pm}$ of $M$, which are the roots of \eqref{charpoly}:
\[
  \lambda^{}_{\pm} \, = \, \frac{1}{2}\, \bigl(T \pm \sqrt{T^2-4D}\,\bigr) 
  \, = \, \frac{1}{2}\, \bigl(T \pm \sqrt{\!\varDelta}\,\bigr) \, = \, 
  \frac{1}{2}\, \bigl(T \pm \mgcd(M) \sqrt{\!\varDelta'}\,\bigr) \, = \,
   \frac{1}{2}\, \bigl(T \pm S \sqrt{\!\varDelta''}\,\bigr).
\]   
In the last equality, $\ts\!\varDelta''$ is the square free part of
$\varDelta$, highlighting that $\mgcd(M)\ts |\ts S$. In particular,
$\mgcd(M)=1$ when $\ts\!\varDelta$ itself is square free.  The
approach of \cite{PV} to studying the dynamics of $\GL (2,\ZZ)$
matrices on $L_n$ involves relating it to the multiplication of the
eigenvalue $\lambda_{+}$ on an associated ideal. The latter is an
order in a quadratic number field, but generally not its maximal
order; see \cite[Ch.~2]{BS} and \cite{Z} for details on the connection
between quadratic forms and orders in number fields.  \exend
\end{remark}

\subsection{Complete determination of local conjugacy}

We now give a complete description of the local conjugacy problem, in
the sense of Definition~\ref{defn-local}. This makes use of a
particular normal form over the ring $\ZZ_{n}$, with a sequence of
elementary arguments.

\begin{prop} \label{prop-4}
  Let $M\in\Mat (2,\ZZ)$ be a matrix with $\mgcd (M)=\mu \neq 0$, and
  let $C$ be the corresponding companion matrix, as in
  $\eqref{mat-defs}$.  Then, for all integers $n\ge 2$ that are
  relatively prime with $\mu$, the reductions of the matrices $M$ and
  $C$ mod $n$ are $\Mat (2,\ZZ_n)^{\times}$-conjugate. In this case,
  $M$ and $C$ share the same local statistics on $L_n$.
\end{prop}
\begin{proof}
  With $\mgcd (M)=\mu$, the integer quadratic form
  $\frac{1}{\mu}Q^{}_{M}$ is primitive.  By Fact~\ref{prim-rep}, we
  can thus find $\alpha,\beta\in\ZZ$ with $\gcd (\alpha,\beta)=1$ and
  $\frac{1}{\mu}Q^{}_{M} (\alpha,\beta)=k$ where $k$ is an integer
  relatively prime with $n$. So, $Q^{}_{M} (\alpha,\beta)=k\mu$, which
  is still relatively prime with $n$ by assumption.

  This means that there is a matrix $A\in\Mat (2,\ZZ)$ with $\det
  (A)=k\mu$ whose reduction mod $n$ is an element of $\Mat
  (2,\ZZ_n)^{\times}$. Consequently, it defines a conjugacy of $M$ and
  $C$ on $\tilde{L}_n$, again when taking all matrices mod $n$.  The
  final claim is clear from Fact~\ref{conj-implies-stats}.
\end{proof}

Note that we need not consider the trivial lattice, $L_1 = \{0\}$, as
the point $x=0$ is fixed by all matrices.

\begin{example}\label{sl-versus-gl}
  The conjugation in Proposition~\ref{prop-4} need not be
  $\SL(2,\ZZ_n)$-conjugacy, that is, the conjugating element need not
  have determinant $1$ (mod $n$). To illustrate this, consider the
  matrices
\[
  M = \begin{pmatrix} 2 & 3 \\ 2 & 2 \end{pmatrix}
  \quad \text{and} \quad
  C = \begin{pmatrix} 0 & 2 \\ 1 & 4 \end{pmatrix},
\]
with $D=-2$, $T=4$ and, from \eqref{det-form}, quadratic form
$Q^{}_{M} (\alpha,\gamma) = 2\alpha^2 - 3\gamma^2$ (primitive since
$\mgcd (M) =1$). Over $\ZZ^{}_{3}$, the quadratic form reduces to
$2\alpha^2$, which can represent $0$ and $2$, but never $1$, so that
$M$ and $C$ are not conjugate over $\SL (2,\ZZ^{}_3)$.  This example
highlights that the approach of \cite{AO2} is of limited value for our
problem.
\exend
\end{example}

The following result is obvious from Proposition~\ref{prop-4}, because
$\Mat (2,\ZZ_n)^{\times}$-conjugacy induces a graph isomorphism in the
sense of Definition~\ref{defn-local}, compare
Fact~\ref{conj-implies-stats}.

\begin{coro}\label{same-local-stats}
  Let $M,M^{\ts\prime}$ be two integer matrices with the same trace
  and determinant, whence they share the same companion matrix. If, in
  addition, $\mgcd (M^{\ts\prime}) = \mgcd (M) =1$, the toral
  endomorphisms defined by $M$ and $M^{\ts\prime}$ have the same local
  statistics on all lattices $L_n$.  \qed
\end{coro}

\begin{example}\label{not-conjugate-more}
Let us briefly return to the matrices $M$, $M^{-1}$ and $C$ of 
Example~\ref{not-conjugate}. They all have $\mgcd=1$, but are in
different $\GL (2,\ZZ)$ conjugacy classes. Nevertheless,
Corollary~\ref{same-local-stats} shows they all give the same
local statistics, on all lattices $L_n$.   
The same conclusion applies to the matrices $M$ and $M^{\ts\prime}$
of Example~\ref{old-ex}.
\exend
\end{example}

Let us proceed to the general case when $\mgcd (M)=r \ge 1$ from
the point of view of local conjugacies.  
Consider the matrix $M$ from \eqref{two-matrices}, and let $r=\mgcd
(M)$. We may now decompose $M$ as
\begin{equation} \label{decomposition}
   M \, = \, a\ts\one + r\ts N
   \quad \text{with} \quad N \, = \, \begin{pmatrix}
    0 & \tilde{b} \\ \tilde{c} & (d-a)^{\up} \end{pmatrix}.
\end{equation}
When $r=0$, we have the trivial case $M=a\ts\one$, which we
now put aside by assuming $r\in\NN$, in line with Fact~\ref{exclude},
so that $\mgcd (N)=1$ by construction.

\begin{prop} \label{normal-form}
  Let $M=\left(\begin{smallmatrix}a&b\\c&d\end{smallmatrix}\right)$ be
  an integer matrix with $\mgcd (M)=r\neq 0$. Then, for all integers
  $n\ge 2$, and when taking all matrices mod $n$, $M$ is $\Mat
  (2,\ZZ_n)^{\times}$-conjugate to the integer matrix
  $\left(\begin{smallmatrix}a & bc/r\\r & d\end{smallmatrix}\right)$.
\end{prop}
\begin{proof}
  By construction, the matrix $N$ in the decomposition
  \eqref{decomposition} is still an integer matrix, with $\mgcd (N) =
  \gcd \bigl(\tilde{b},\tilde{c}, (d-a)^{\up}\ts\bigr)=1$.  By
  Proposition~\ref{prop-4}, for all integers $n\ge 2$, $N$ is then
  $\Mat (2,\ZZ_n)^{\times}$-conjugate to its companion matrix
\[
     C^{}_{N} \, = \, \begin{pmatrix}
     0 & \tilde{b}\,\tilde{c} \\ 1 & (d-a)^{\up}
     \end{pmatrix} ,
\]
where both $N$ and $C^{}_{N}$ are taken mod $n$.

The claim now follows from the observation that this conjugacy extends
to the one claimed by the structure of the decomposition $M=a\ts\one +
r\ts N$. Indeed, the relation $N=A C^{}_{N} A^{-1}$ together with
\eqref{decomposition} immediately implies $AMA^{-1} = a\ts\one + r\ts
C^{}_{N}$.
\end{proof}

It is worth recalling that \cite[Lemma~1]{B} does not apply to
$2\!\times\! 2$-matrices, meaning that we do not have global conjugacy in
general.  The benefit of Proposition~\ref{normal-form} is that we
obtain a local conjugacy instead, and to a matrix with a well-defined
element in the lower left corner. For the further arguments, we need a
technical result, formulated within $\Mat (2,\ZZ)$, hence \emph{prior}
to looking at the reductions mod $n$.
\begin{lemma} \label{technical-lemma} 
  Let $M,M^{\ts\prime}\in\Mat(2,\ZZ)$ be the two matrices of
  $\eqref{two-matrices}$ and assume that
  $\det(M)=\det(M^{\ts\prime})=D$ and
  $\trace(M)=\trace(M^{\ts\prime})=T$.  Assume, in addition, that
  $\mgcd(M)=\mgcd(M^{\ts\prime})=r\in\NN$. Then, $r$ divides
  $d-d{\ts}'$.
\end{lemma}
\begin{proof}
When $d=d{\ts}'$, the statement is trivially true, so we may assume
that $d\neq d{\ts}'$.  Since
$r=\gcd(b,c,d-a)=\gcd(b{\ts}',c{\ts}',d{\ts}'-a{\ts}')$, one has $r^2
\ts | \ts bc$ and $r^2\ts | \ts b{\ts}'c{\ts}'$.  Observe that
\[
   (a-d{\ts}')\ts (d-d{\ts}') \, = \, ad + (d{\ts}'-T)\ts d{\ts}'
   \, = \, ad-a{\ts}'d{\ts}' \, = \, D+bc - D-b{\ts}'c{\ts}'
   \, = \, bc-b{\ts}'c{\ts}' \ts ,
\]
which implies that $r^2\ts | \ts (a-d{\ts}')\ts (d-d{\ts}')$.
With $m:= d-d{\ts}'\neq 0$ and $d-a=k\ts r$, where $k\in\ZZ$
by assumption, we now have to consider
\begin{equation} \label{square-divide}
     r^2 \mid m\ts (m-k \ts r)
\end{equation}
and to show that this implies $r\ts | \ts m$. When $m=kr$, this
is again clear, so assume $m-kr\neq 0$. Then, let $r^2 =
p_{1}^{2s^{}_{1}}\cdot\ldots\cdot p_{\ell}^{2s^{}_{\ell}}$ be the
unique prime decomposition of $r^2$ into powers of distinct
primes. Let $t_i$ be the highest power so that $p_{i}^{t_i}\ts | \ts
m$, which is a non-negative integer.

Assume that $t_i < s_i$ for some index $i$. By \eqref{square-divide},
this implies $p_{i}^{2s_i - t_i}\ts | \ts (m-k\ts r)$ and thus, as
$p_{i}^{s_i}\ts |\ts r$ and $2s_i - t_i > s_i$, also $p_{i}^{s_i}\ts |
m$, in contradiction to the assumption that $t_i < s_i$. Consequently,
we must indeed have $t_i\ge s_i$ for all $1\le i\le \ell$. Since
$r=\prod_{i} p_{i}^{s_i}$, this means that $r\ts | \ts m$ as claimed.
\end{proof}

\begin{prop} \label{local-complete} 
    Let $M,M^{\ts\prime}$ be the matrices from $\eqref{two-matrices}$
  and assume that they have the same trace and determinant.  Assume
  further that $\mgcd (M) = \mgcd (M^{\ts\prime})=r$.  Then, for an
  arbitrary integer $n\ge 2$, the reductions mod $n$ of the matrices
  $M$ and $M^{\ts\prime}$ are $\Mat (2,\ZZ_n)^{\times}$-conjugate.
\end{prop}
\begin{proof}
  When $r=0$, Fact~\ref{exclude} implies $M=M^{\ts\prime}$. So, let us
  assume $r\in\NN$.  By Proposition~\ref{normal-form}, we know that
  matrices $M$ and $M^{\ts\prime}$ are $\Mat
  (2,\ZZ_n)^{\times}$-conjugate to integer matrices $N$ and $N'$,
  where
\[
    N \, = \, \begin{pmatrix} a & \frac{b\, c}{r}\\
      r & d \end{pmatrix}  \quad \text{and} \quad
    N{\ts}' \, = \, \begin{pmatrix} a{\ts}' & 
       \frac{b{\ts}' c{\ts}'}{r}\\
      r & d{\ts}' \end{pmatrix} ,
\]
in the sense that their reductions mod $n$ satisfy the corresponding
conjugacies.  Our claim follows if we can show that also $N$ and $N'$
are $\Mat (2,\ZZ_n)^{\times}$-conjugate in this sense.

Consider the unimodular matrix $A=\left(\begin{smallmatrix}1 &
\frac{d-d{\ts}'}{r} \\ 0 & 1 \end{smallmatrix}\right)$, which is an
integer matrix by Lemma~\ref{technical-lemma} and hence an element of
$\GL (2,\ZZ)$. Using $\trace (M)=\trace (M{\ts}')$ and $\det (M)=\det
(M{\ts}')$, it is easy to check that $A N = N' A$ holds, hence
$N'=ANA^{-1}$ within $\GL (2,\ZZ)$. By Fact~\ref{global-to-local}, this
implies the local conjugacy claimed.
\end{proof}

Let us investigate the local conjugacies a bit further,
aiming at a converse of Proposition~\ref{local-complete}.
\begin{prop} \label{local-mod-n} 
  Let $M,M^{\ts\prime}\in\Mat (2,\ZZ)$ be two integer matrices whose
  reductions mod $n$ are $\Mat (2,\ZZ_n)^{\times}$-conjugate for some
  $n\ge 2$. Then, $M$ and $M^{\ts\prime}$ have the same determinants
  and traces mod $n$, and their matrix gcds $r,r'$ generate the same
  ideal in $\ZZ_n$, meaning $r\ts \ZZ_n = r'\ts \ZZ_n$.
\end{prop}
\begin{proof}
  The statement about the determinants and traces is clear. For the
  claim about the matrix gcd, we can again use the idea of the proof
  of Lemma~\ref{mat-gcd}, up to the point where we conclude that $r\ts
  |\ts r'$ and $r'\ts |\ts r$, now seen as divisibility properties
  within $\ZZ_n$. But $k\ts |\ts\ell$ means that the principal ideal
  $(\ell)=\ell\ZZ_n$ is contained in $(k)$, so that we obtain
  $(r)\subset (r')$ and $(r')\subset (r)$, hence $(r)=(r')$.
\end{proof}

At this stage, one can formulate the following central result.

\begin{theorem}\label{main-theorem}
  {}For two integer matrices $M,M^{\ts\prime}\in\Mat (2,\ZZ)$, the 
  following statements are equivalent:
\begin{itemize}
\item[{\rm (a)}]  The reductions mod $n$ of $M$ and $M^{\ts\prime}$ 
   are $\Mat (2,\ZZ_n)^{\times}$-conjugate for all $n\ge 2$;
\item[{\rm (b)}] $M$ and $M^{\ts\prime}$ satisfy the three relations
   $\det(M)=\det(M^{\ts\prime})$, $\trace(M)=\trace(M^{\ts\prime})$ and
   $\mgcd(M)=\mgcd(M^{\ts\prime})$.
\end{itemize}
\end{theorem}
\begin{proof}
  The direction (b) $\Longrightarrow$ (a) follows directly from
  Proposition~\ref{local-complete}.  

  For the converse direction, we may conclude from
  Proposition~\ref{local-mod-n} that
  $\det(M)\equiv\det(M^{\ts\prime})$ mod $n$ and
  $\trace(M)\equiv\trace(M^{\ts\prime})$ mod $n$ for all $n\ge 2$.
  Consequently, we must have $\det(M)=\det(M^{\ts\prime})$ and
  $\trace(M)=\trace(M^{\ts\prime})$ (recall that $k\equiv\ell$ mod
  $n\in\NN$ means that $k-\ell$ is divisible by $n$, which 
  simultaneously holds for all $n\in\NN$ only when $k-\ell=0$).

  For the third identity, let $r=\mgcd(M)$ and
  $r^{\ts\prime}=\mgcd(M^{\ts\prime})$ and assume that $r\ts\ZZ_n =
  r'\ts\ZZ_n$ for all $n\ge 2$, but $r\neq r^{\ts\prime}$.
  Consequently, there is a prime $p$ with $r = p^s \varrho$ and
  $r^{\ts\prime} = p^t \varrho^{\ts\prime}$, $t\neq s$, such that
  $\varrho$ and $\varrho^{\ts\prime}$ are both relatively prime with
  $p$. Without loss of generality, we may assume that $t>s$, and then
  choose $n=p^t$.  Clearly, both $\varrho$ and $\varrho^{\ts\prime}$
  are then units in $\ZZ_n$ by construction.  With $(r):=r\ts\ZZ_n$,
  we then obtain $(r)=(p^t \varrho)=(p^t)=(0)$, while
  $(r^{\ts\prime})=(p^s \varrho^{\ts\prime})=(p^s)\neq (0)$, in
  contradiction to the original assumption.
\end{proof}

This theorem permits the following answer to the question
for the local statistics of an integer matrix.

\begin{coro}\label{main-coro}
  The complete local statistics of an integer matrix $M$, in the
  sense of Definition~$\ref{defn-local}$, only depends
  on the three invariants $\det (M)$, $\trace (M)$ and
  $\mgcd (M)$. Two integer matrices with the same triple
  of invariants thus have the same local statistics on all lattices
  $L_n$. In particular, they have the same fixed point counts, both
  locally and globally.  \qed
\end{coro}

The result of Theorem~\ref{main-theorem} permits an interpretation in
terms of $\widehat{\ZZ}$, which is the inverse (or projective) limit
of the rings $\ZZ_n$ over the positive integers ordered by
divisibility, written as $\widehat{\ZZ} = \varprojlim \ZZ_n$. It is
also known as the Pr\"{u}fer ring, see \cite{Neu} for details.

\begin{coro}
  In the situation of Theorem~$\ref{main-theorem}$, any of the two
  conditions is equivalent to the matrices $M$ and $M^{\prime}$ being
  $\GL (2,\widehat{\ZZ})$-conjugate.
\end{coro}  
\begin{proof}
 Clearly, a $\GL (2,\widehat{\ZZ})$-conjugacy implies condition (a) of 
 Theorem~\ref{main-theorem}. Conversely, assuming (a), consider the
 inverse system of $\Mat (2,\ZZ_n)^{\times}$-subsets $X(n)$ defined by
\[
     X(n) = \{ A \in \Mat (2,\ZZ_n)^{\times} \mid 
        A M = M^{\ts\prime} \! A \}
\] 
with $n\in\NN$ and ordered inductively by divisibility. All $X(n)$ are
non-empty by assumption and finite. Let $X=\varprojlim X(n)$ be the
inverse limit, which is then non-empty as well. Any element in $X$
achieves the required conjugacy via the corresponding projections.
\end{proof}

We conclude by noting:

\begin{remark} \label{always-rev}
  One consequence of Theorem~\ref{main-theorem} for $\SL (2,\ZZ)$
  matrices, noting \eqref{gcd-on-inverse}, is that such a matrix is
  conjugate to its inverse on \emph{all} lattices individually,
  despite the fact that this is generally untrue on the torus, as in
  Example~\ref{not-conjugate}. This relates to the investigations of
  quantum cat maps and their perturbations in \cite{KM}.  \exend
\end{remark}

\begin{remark} \label{spectrumconj} 
  We have described the three invariants for complete $\Mat
  (2,\ZZ_n)^{\times}$-conjugacy, whereas \cite{ATW} presented the
  analogous result for $\GL (2,\ZZ)$-conjugacy.  The conjugacy of $M$
  and $M^{\ts\prime}$ in the special linear group over the $p$-adic
  integers is a related question, addressed in
  \cite{AO2}, though Example~\ref{sl-versus-gl} above indicates
  the difference to $\GL (2,\ZZ_n)$-conjugacy.  \exend
\end{remark}

\smallskip
\section*{Acknowledgements}
It is a pleasure to thank Grant Cairns, Natascha Neum\"{a}rker, Franco
Vivaldi, Benjy Weiss and Peter Zeiner for discussions and useful hints
on the manuscript.  This work was supported by the German Research
Council (DFG), within the CRC 701, and under the Australian Research
Council's Discovery funding scheme (project number DP$\,$0774473),
including mutual visits.

\end{document}